\documentclass{amsart}
\RequirePackage{amsmath}
\RequirePackage{amssymb}
\RequirePackage{amsxtra}
\RequirePackage{amsfonts}
\RequirePackage{latexsym}
\RequirePackage{euscript}
\RequirePackage{amscd}
\RequirePackage{amsthm}
\RequirePackage{xypic}
\usepackage{stmaryrd}
\usepackage{mathtools}
\xyoption{all}

\usepackage{tikz}
\makeatletter
\newlength\xvec@height%
\newlength\xvec@depth%
\newlength\xvec@width%
\newcommand{\xvec}[2][]{%
  \ifmmode%
    \settoheight{\xvec@height}{$#2$}%
    \settodepth{\xvec@depth}{$#2$}%
    \settowidth{\xvec@width}{$#2$}%
  \else%
    \settoheight{\xvec@height}{#2}%
    \settodepth{\xvec@depth}{#2}%
    \settowidth{\xvec@width}{#2}%
  \fi%
  \def\xvec@arg{#1}%
  \def\xvec@dd{:}%
  \def\xvec@d{.}%
  \raisebox{.2ex}{\raisebox{\xvec@height}{\rlap{%
    \kern.05em
    \begin{tikzpicture}[scale=1]
    \pgfsetroundcap
    \draw (.05em,0)--(\xvec@width-.05em,0);
    \draw (\xvec@width-.05em,0)--(\xvec@width-.15em, .075em);
    \draw (\xvec@width-.05em,0)--(\xvec@width-.15em,-.075em);
    \ifx\xvec@arg\xvec@d%
      \fill(\xvec@width*.45,.5ex) circle (.5pt);%
    \else\ifx\xvec@arg\xvec@dd%
      \fill(\xvec@width*.30,.5ex) circle (.5pt);%
      \fill(\xvec@width*.65,.5ex) circle (.5pt);%
    \fi\fi%
    \end{tikzpicture}%
  }}}%
  #2%
}
\makeatother

\usepackage{dsfont}

\def\GL{\operatorname{GL}}

\def\Gal{\mathrm{Gal}}

\swapnumbers

\newlength{\ownl}

\newcommand{\Symm}{{\operatorname{Symm}\,}}

\newcommand{\barF}{\overline{{F}}}

\newcommand{\barr}{\overline{{r}}}

 \newcommand{\barmu    }{\overline{\mu}}

\def\RCS$#1: #2 ${\expandafter\def\csname RCS#1\endcsname{#2}}
\RCS$Revision: 85 $
\RCS$Date: 2011-12-16 18:54:17 -0600 (Fri, 16 Dec 2011) $
 
\usepackage{hyperref} 

\newtheorem{thm}[section]{Theorem}

\theoremstyle{definition}

\theoremstyle{remark}

\def\numequation{\addtocounter{subsection}{1}\begin{equation}}
\def\nummultline{\addtocounter{subsubsection}{1}\begin{multline}}
\def\anumequation{\addtocounter{subsection}{1}\begin{equation}}

\title{Automorphy lifting for small $l$ - Appendix B to ``Automorphy of $\Symm^5(\GL(2))$ and base change''}
\author{Luis Dieulefait} \email{ldieulefait@ub.edu} \address{Facultat
  de Matemàtiques, Universitat de Barcelona}
\author{Toby Gee} \email{toby.gee@imperial.ac.uk} \address{Department of
  Mathematics, Imperial College London}
\thanks{The second author was supported in part by a Marie
  Curie Career Integration Grant.}

\begin{document}
\maketitle

In this appendix we prove a slight generalization of Theorem 4.2.1
of~\cite{BLGGT}. It strengthens \emph{loc. cit.} in that it weakens
the assumption that $l\ge 2(n+1)$ to an adequacy hypothesis (which is
automatic if $l\ge 2(n+1)$ by the main result of~\cite{jackapp}).

This theorem can be proved by a straightforward modification of the
proof of Theorem 4.2.1 of~\cite{BLGGT}, using Lemma A.3.1
of~\cite{blggU2} (which was proved by Richard Taylor during the
writing of~\cite{BLGGT}). In order to make the proof straightforward
to read, rather than explaining how to modify the proof of Theorem
4.2.1 of~\cite{BLGGT} using this Lemma, we combine Theorem A.4.1
of~\cite{blggU2} (which is an improvement on Theorem 4.3.1
of~\cite{BLGGT} in exactly the same way that Theorem \ref{mainmlt}
below is an improvement on Theorem 4.2.1 of~\cite{BLGGT}) with Theorem
2.3.1 of~\cite{BLGGT} (which is essentially Theorem 7.1
of~\cite{jack}).

We freely use the notation and terminology of~\cite{BLGGT} without
comment. We would like to thank Florian Herzig for his helpful
comments on an earlier version of this appendix.
        \begin{thm} \label{mainmlt} 
     Let $F$ be an imaginary CM field with maximal totally real subfield $F^+$ and let $c$ denote the non-trivial element of $\Gal(F/F^+)$. Suppose that $l$ is an odd prime, and that $(r,\mu)$ is a regular algebraic, irreducible, $n$-dimensional, polarized representation of $G_F$. 
Let $\barr$ denote the semi-simplification of the reduction
of $r$. 
Suppose that $(r,\mu)$ enjoys the following properties:
   \begin{enumerate}
\item\label{pdiag} $r|_{G_{F_v}}$ is potentially diagonalizable (and so in particular potentially crystalline) for all $v|l$.
\item  The restriction $ \barr(G_{F(\zeta_l)})$ is adequate, and $\zeta_l \not\in F$.
\item\label{auto}  $(\barr,\barmu)$ is either ordinarily automorphic or potentially diagonalizably automorphic.
\end{enumerate}

Then $(r,\mu)$ is potentially diagonalizably automorphic (of level potentially prime to $l$).  
\end{thm}
\begin{proof} By Lemma 2.2.2 of~\cite{BLGGT} (base change) 
   it is enough to prove the theorem after replacing $F$ by a soluble
   CM extension which is linearly disjoint from $\barF^{\ker
     \barr}(\zeta_l)$ over $F$. Thus we can and do suppose that  all primes dividing $l$ and all primes at which  $r$ ramifies are split over $F^+$.

Suppose firstly that we are in the case that
  $(\barr,\barmu)$ is ordinarily automorphic. Then by Hida theory we
  may reduce to the potentially diagonalizably automorphic case,
  because by (for example) Lemma 3.1.4 of~\cite{gg} and Lemma 1.4.3(1)
  of~\cite{BLGGT}, every Hida family passes through points for which
  the associated $l$-adic Galois representation is potentially
  diagonalizable at all places dividing $l$.

Suppose now that we are in the case that $(\barr,\barmu)$ is
  potentially diagonalizably automorphic. Applying Theorem A.4.1
  of~\cite{blggU2} (with $F=F'$, and the $\rho_v$ being $r|_{G_{F_v}}$), we see that there is a regular
  algebraic, cuspidal,  polarized automorphic representation $(\pi,\chi)$ of level potentially prime to $l$, such that
  \[ r_{l,\imath}(\pi)|_{G_{F_v}}\sim r|_{G_{F_v}} \] for each finite
  place $v$ of $F$. The result then follows immediately from Theorem
  2.3.1 of~\cite{BLGGT}.
\end{proof}
\bibliographystyle{amsalpha}
\bibliography{dieulefaitgee}
\end{document}